\documentclass[12pt, a4paper]{amsart}
\usepackage{amsfonts,amsmath, amsthm, amssymb}
\usepackage{latexsym}
\usepackage{enumitem}

\usepackage[dvips]{graphicx}
\usepackage{psfrag}
\usepackage{xcolor}
\usepackage{tikz}
\usetikzlibrary{shapes,arrows,positioning,intersections}
\usepackage{hyperref}

\newtheorem{theorem}{Theorem}[section]
\newtheorem{lemma}[theorem]{Lemma}
\newtheorem{corollary}[theorem]{Corollary}
\newtheorem{proposition}[theorem]{Proposition}

\newtheorem{definition}[theorem]{Definition}
\newtheorem{remark}[theorem]{Remark}
\newtheorem{example}[theorem]{Example}

\newcommand{\marginnotem}[1]{\mbox{}\marginpar{\tiny\raggedright\hspace{0pt}{\color{magenta} {\bf Mark}$\blacktriangleright$  #1}}}

\newcommand{\useonX}{} 
\newcommand{\useonM}{\underline}

\newcommand{\tm}{\useonX{m}}
\newcommand{\tB}{\useonX{B}}
\newcommand{\tS}{\useonX{S}}
\newcommand{\tc}{\useonX{c}}
\newcommand{\tv}{\useonX{v}}
\newcommand{\tw}{\useonX{w}}
\newcommand{\hB}{\useonM{B}}
\newcommand{\hS}{\useonM{S}}

\newcommand{\hm}{\useonM{m}}

\newcommand{\Pa}{\mathbf{P}}
\newcommand{\Fu}{\mathbf{F}}

\newcommand{\sqbd}{\partial^2 X}
\newcommand {\R}{\mathbb{R}} 

\DeclareMathOperator{\diam}{diam}

\DeclareMathOperator{\pr}{pr}

\makeatletter
\@namedef{subjclassname@2020}{\textup{2020} Mathematics Subject Classification}
\makeatother

\begin{document}

\title{Dimension of the Feigenbaum attractor}

\author{Mark Pollicott}


\maketitle

\begin{abstract}
In this note we propose an effective method to estimate the dimension of the Feigenbaum attractor for the period doubling phenomenon.   In particular, we will describe a way to convert the highly accurate estimates for $g$ into better estimates on $\dim(X)$.
\end{abstract}

\section{Introduction}
The original experimental  discoveries and conjectures  of  Feigenbaum \cite{feigenbaum}
and Coullet-Tresser \cite{ct} in the mid-1970s have been  an important catalyst in the development of the  general theory of renormalization in dynamical systems.
These important  empirical results described the period doubling bifurcations of families of unimodal maps and presented a framework for understanding the underlying mechanism.
The majority  of these  influential conjectures  were subsequently proved in 1982 by Lanford, with later profound advances  by H. Epstein, Lyubich, McMullen,  Sullivan  and others. 
 
 \begin{theorem}[Feigenbaum Conjectures :  Lanford's Theorem]
 There  exists a $C^\omega$  map $g: [-1,1] \to [-1,1]$
 such that:
\begin{enumerate}
\item $g(0) = 1$;
\item  $g$ is unimodal  (i.e., a functions which is first monotone increasing for $x<0$ and then monotone decreasing for $x>0$ after passing the critical point at $0$ with $g'(0) = 0$ and $g''(0) <  0$); and 
\item   $g$ is symmetric (i.e., $g(x) = g(-x)$),
\end{enumerate}
such that $g$ satsifies the \emph{ Cvitanovic-Feigenbaum Functional equation} $\mathcal R(g) = g$ where 
 $$
\mathcal R g(x) = \alpha g \circ g \left(\frac{x}{\alpha} \right)
\hbox{ and }
\alpha =
  -1/f(1). \eqno(1.1)$$
 \end{theorem}
 
 \begin{figure}[h!]
  \includegraphics[width=0.4\textwidth]{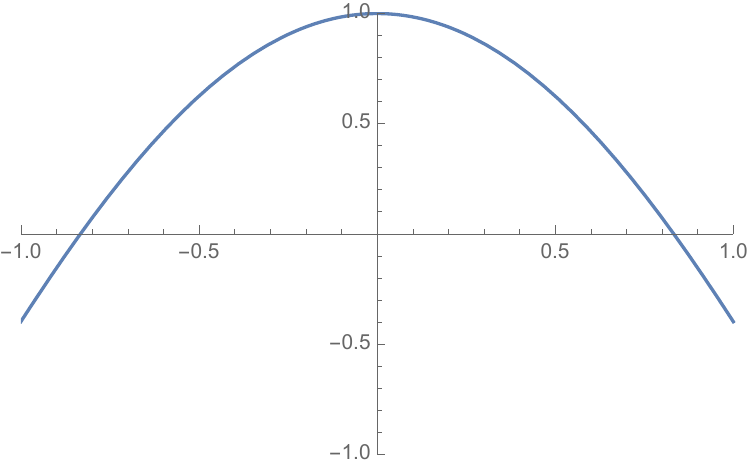}
  \hskip 1cm
    \includegraphics[width=0.4\textwidth]{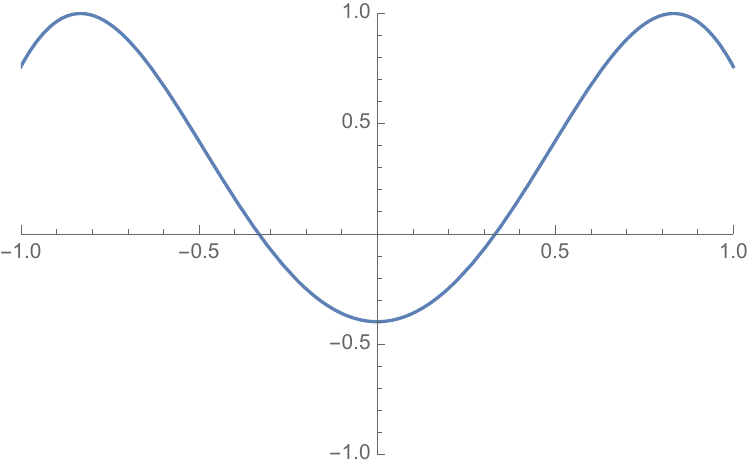}
 \caption{(i) A plot of the  unimodal $g$ which is a fixed point for the Feigenbaum-Cvitanovic functional equation (1);  (ii) A plot of  $g\circ g$, the middle portion of which looks like an inverted and rescaled plot of $g$.
 }
 \end{figure}
 
 \noindent 
In 
Lanford's original  proof he  constructed a power series solution
$$
g(x) =  1+ \sum_{n=1}^\infty a_{2n} x^{2n}. \eqno(1.2)
$$
The solution was based on the use of fixed point theorems applied to a suitable family of holomorphic functions on an explicit  domain.
The proof is computer assisted and 
 many of the values for $a_m$ can be  computed  to high precision (e.g., for $m \leq 170$, say, the $a_m$ are computed to in excess of  $150$  decimal places in \cite{mathar}). 
 
  \subsection{Numerical invariants}
Associated to the solution $\mathcal R (g) = g$ are a number of interesting  numerical values (dubbed ``Feigenvalues'').
The best known such value is probably
$$\delta = 
4.66920160910299067185320382046620161  \ldots$$
 which played an important role in Feigenbaum's original empirical discoveries for the logistic map $F_mu(x) = 1 - \mu x^2$ converge.  More precisely, if $\mu_1 < \mu_2 < \mu_3 < \cdots $ are the  parameter values for period doubling then 
 $$
 \frac{\mu_n-\mu_{n-1}}{\mu_{n+1} - \mu_n} \to \delta \hbox{ as } n \to +\infty.
 $$
  A second value is  
  $$\alpha = -\frac{1}{g(1)} = -2.5029078750958928222839028732182157\ldots$$
  which can be  characterized by the (signed) distances
   $(d_n)_{n=1}^\infty$ of the  $2^n$-attracting  orbit  from the origin  at those  parameter values $\mu$ for which $0$ is superattracting.  More precisely,    
    $$
 \frac{d_n}{d_{n+1}} \to \delta \hbox{ as } n \to +\infty.
 $$
 
The constants $\delta$ and $\alpha$ are universal in that the same constant arises for any similar family of unimodal maps.  
 Furthermore,  these values  can be extracted from the function $g$ and thus 
 a  detailed knowledge of the coefficients $(a_m)$ in (1.2) allows  them    to be computed to high precision. 
  
 Finally, we want to consider a third  numerical value associated to $g$.
We first need the following definition.

\begin{definition}
The \emph{Feigenbaum attractor}
$X = X(g)  \subset [-1,1]$ 
 is the $g$-invariant Cantor set which is the closure of the orbit
  of the critical point, i.e., $X = \overline{\cup_{n=0}^\infty g^n(0)}$.
\end{definition}
The restriction $g: X \to X$ is topologically conjugate to a dyadic odometer (or adding machine).

A natural numerical invariant is  the Hausdorff dimension $\dim(X) \in (0,1)$ of the attractor $X$.
Again, this  is a universal value in the following sense:
 
 \begin{theorem}[Coullet-Tresser geometric rigidity of Cantor sets \cite{ct}, \cite{sullivan}]
  Any quadratic map $f$ whose critical point has an orbit with the same combinatorics 
  as $g$ has an attracting Cantor set $X(f)$ with Hausdorff dimension $ \dim X(f) = \dim(X(g))$.
 \end{theorem}

   Unfortunately, the numerical value of the Hausdorff dimension $\dim(X)$ is  notoriously difficult to rigourously estimate  despite  the  map $g$ being known  to great accuracy.   The purpose of this note is to describe a way to convert the highly accurate estimates for $g$ into better estimates on $\dim(X)$.

\subsection{Estimates on $\dim(X)$}
The difficulty is in finding an efficient way to convert a detailed knowledge of the coefficients $(a_m)$ for $g$ into good estimates on $\dim(X)$.
Estimates for the Hausdorff dimension $\dim(X)$ of the attractor were given:   by  Grassberger \cite{grassberger2} (to $8$ decimal places),
Bensimon, Jensen and Kadanoff \cite{bensimon} (to $10$ decimal places) and Kovacs.  However, the    best {\it non-rigorous}  (albeit computationally stable) numerical result to date is  due Christiansen et al  \cite{Christiansen} given to  $27$ decimal places:
$$\dim(X) = 0.538 045 143 580 549 911 671 415 56 \ldots \eqno(1.3)$$
This estimate is based on the method of cycle expansions which depends on periodic points for $g$ whose exponential growth makes it difficult to make significant improvements using this method.  
 Moreover, it is difficult to get effective error bounds (which is inherent in  \cite{jp}).
Other estimates due to Grassberger  \cite{grassberger2} 
are  described in  (\cite{rasband}, pp.80-81).

There are difficulties in getting rigorous error bounds. However, using the simple  on the derivatives of $T_1'$ and $T_2$, there is a basic bound in
(\cite{falconer}, p.141) of 
$0.5345 
< \dim(X) < 
 0.5544$.
More recently, Burbanks, Osbaldstein and Thurlby \cite{BOT}  refined this approach to give  bounds
$$
0.53705 \ldots
 < \dim(X)  <
 0.53917 \ldots
$$
which they claim to be the best rigorous values.

In this note we will use a different method to estimate $\dim(X)$ which appears to have the advantage that it is quite successful in exploiting  the very precise knowledge of the values $(a_n)$ to automatically give more accurate rigorous estimates for $\dim(X)$
Our main result is the following.
\footnote{Subject to  transcription errors for the values of $a_m$ and accepting the accuracy of the Mathematica computations}

\begin{theorem}\label{estimate}
We have the  estimate




$$
\begin{aligned}
\dim(X) &=  0.53804 51435 80549 91167 14155 67374 98629 2737\cr
&\qquad 9 64965 87785  \cdots  \pm 10^{-50}
\end{aligned} \eqno(1.4)
$$
accurate to $52$ decimal places. \end{theorem}

 In particular, this confirms  the non-rigorous   estimate  of Christiansen et al  in  (1.3). 
 
The method of proof of the bound in Theorem \ref{estimate} is based on a general minmax approach which seems to be more effective in converting the very precise knowledge of $g$ and the coefficients $a_m$ into estimates on $\dim(X)$.

\section{Iterated function schemes and Dimension}

We begin with some useful background material.

  \subsection{An iterated function scheme}
There is a well known alternative  construction of $X = X(g)$ as the limit set of an iterated function scheme.  More precisely, we want to use  the following lemma.
 
\begin{lemma}\label{ifs}
The Feigenbaum attractor is the limit set of the  iterated function scheme consisting of two  contractions
$T_1, T_2: [1/\alpha, 1] \to [1/\alpha, 1]$ given by 
$$
T_1(x) = \frac{x}{\alpha}
\hbox{ and }
T_2(x) = g^{-1}\left(\frac{x}{\alpha}\right)
$$
(i.e., $X$ is the smallest non-empty closed set such that $X = T_1(X) \cup T_2(X)$). 
\end{lemma}

\noindent
See
 \cite{feigenbaum-ifs} and 
(\cite{falconer}, Theorem 8.2.1). 
The existence of a non-empty limit  set associated to  contractions is a consequence of a general result of Hutchinson
(\cite{falconer-gf}, Theorem 9.1). 
 In Lemma \ref{ifs}  we use $g^{-1}$ to denote  the inverse of $g$ restricted to $[g^2(1),1]$.

\subsection{Dimension}
We want to consider  the Hausdorff dimension of $X$ which, in light of its characterization by Lemma \ref{ifs}, coincides with the Box dimension, whose  simpler definition we briefly recall.   

\begin{definition}
Given $\epsilon > 0$ we  let $N(X, \epsilon)$ be the smallest number of open intervals of length $\epsilon > 0$ needed to cover $X$.  We can then write
$$
\dim(X) = \lim_{\epsilon \to 0} - \frac{\log N(X, \epsilon)}{\log \epsilon}.
$$
\end{definition}

We refer reader to \cite{falconer-gf} for more details.

\medskip 
 Upper and lower bounds on the Hausdorff dimension of the attractor $X$ may be obtained from properties
of the iterated function scheme and with rigorous   bounds on the renormalisation fixed-point $f$.
Previous  rigorous approaches (such as \cite{falconer} and \cite{BOT}) tend to be based on approximation of $T_1$ by (piecewise) affine maps.  However, we will describe in the next section a different approach which allows a more effective use of the detailed knowledge of $g$ to get more accurate estimates on $\dim(X)$.

 
 \section{Transfer operators}
In contrast to previous approaches, we will use what we will refer to as a
 {\it  minmax method} to estimate the dimension $\dim(X)$ 
by looking at transfer operators acting on function spaces.
The benefits of this approach will be shown with empirical estimates in the next section.
 
We can consider a family of bounded linear operators $\mathcal L_t$
(for $t \in \mathbb R$) 
on the Banach space 
 $C^1([-1,1])$
of 
continuously differential 
 functions $h:[-1,1] \to \mathbb R$ with the usual norm $\|h\| = \|h\|_\infty +  \|h'\|_\infty$ where 
$$\|h\|_\infty = \sup_{-1\leq x\leq 1}   |h(x)|.$$
The  definition of the operators is the following: 

\begin{definition}\label{transfer}
The \emph{ transfer operators}  $\mathcal L_t : C^1([-1,1]) \to  C^1([-1,1])$ 
(for $t \in \mathbb R$) 
 are 
defined by
$$
\mathcal L_t h(x) = |T_1'(x)|^t h(T_1x) +  |T_2'(x)|^t h(T_2x)
$$
for $h \in C^1([-1,1]) $ and  $x\in [-1,1]$.
\end{definition}

The spectra $\text{\rm sp}(\mathcal L_t)$ (for $t\in \mathbb R$) of the 
 transfer operators have well known and well understood  properties which we briefly summarize as follows.

\begin{lemma}[after Ruelle \cite{ruelle-book}]\label{ruelle}
For each $t\in \mathbb R$, 
\begin{enumerate}
\item
 the operator  $\mathcal L_t : C^1([-1,1]) \to  C^1([-1,1])$  has a simple maximal positive eigenvalue $\lambda(t) > 0$,  and 
 \item
 the rest of the spectrum is contained in a strictly smaller disk (i.e., $\exists \rho < \lambda(t)$ such that $\text{\rm sp}(\mathcal L_t)\setminus \{\lambda(t)\} \subset \{z\in \mathbb C \hbox{ : } |z| \leq \rho\}$). 
 \end{enumerate}
\end{lemma}

The relevance of these operators to the estimation of the Hausdorff dimension  of $X$ is the following simple lemma.

\begin{lemma}\label{dim}
Let  $ t_0 < t_1$ be such that the maximal eigenvalues of the operators $\mathcal L_{t_0}$ and  
$\mathcal L_{t_1}$, respectively,  
satisfy 
$$\lambda(t_1) > 0 > \lambda(t_0).$$  Then  we have that 
$$t_0 < \dim(X) < t_1.$$
\end{lemma}

\begin{proof}
The result follows directly from the following  two facts.
The maximal   eigenvalue $\lambda(t)$  of the operator $\mathcal L_t$ satisfies:
\begin{enumerate}
\item $t \mapsto \lambda(t)$ is a $C^\omega$  strictly monotone decreasing function.
\footnote{The analyticity comes from $\lambda(t)$ being an isolated eigenvalue and analytic perturbation theeory.  The strict monotonicity 
comes from  an explicit form for the first derivative being negative \cite{ruelle-book})}; and 
\item $\lambda( \dim(X)) = 1$ (by the well known Bowen-Ruelle ``pressure formula''  \cite{bowen}, \cite{ruelle}).
\end{enumerate}
\end{proof}

In order to apply the bounds  in Lemma \ref{dim} to obtain practical estimates on $\dim(X)$ we use  the following useful criteria.

\begin{lemma}\label{minmax}
 Let $\mathcal L_t$ (for  $t \in \mathbb R$) be the operators defined in Definition \ref{transfer}.
\begin{enumerate}
\item
Assume that for $t_0 \in \mathbb R$  there exists a strictly  positive $C^1$ function $h_0:[1/\alpha,1] \to \mathbb R^+$ such that 
$$\sup_{\frac{1}{\alpha}\leq x \leq 1} \frac{(\mathcal L_{t_0}h_0)(x)}{h_0(x)} <1\eqno(3.1) $$ then $\lambda(t_0) \leq 1$.
\item
Assume that for $t_1 \in \mathbb R$   there exists a strictly positive  $C^1$ function $h_1:[1/\alpha,1] \to \mathbb R$ such that 
$$\inf_{\frac{1}{\alpha}\leq x \leq 1} \frac{(\mathcal L_{t_1}h_1)(x)}{h_1(x)} >1 \eqno(3.2) $$ then $\lambda(t_1) \geq1$.
\end{enumerate}
\end{lemma}

\begin{proof}
The proof is very short, so we include it for the reader's convenience.
For part (1) we observe that  applying  $\mathcal L_{t_0}$  repeatedly
 gives 
 that 
for any  $x \in [1/\alpha, 1]$:
$$
\cdots
\leq \mathcal L_{t_0}^n h_0(x) \leq 
\cdots
\leq 
\mathcal L_{t_0}^2 h_0(x) 
\leq 
\mathcal L_{t_0} g(x) \leq h_0(x) \eqno(3.3)
$$
by virtue of the  positivity of the operator and    assumption (3.1).
However,  the spectral properties described in  Lemma \ref{ruelle} imply that 
 $$\lim_{n \to +\infty } \|\mathcal L_{t_0}^n h_0\|_\infty^{1/n} = \lambda(t_0).$$  Moreover, since $h_0>0$ we trivially have 
$\lim_{n \to +\infty } \| h_0\|_\infty^{1/n} = 1$.   The  conclusion of part (1) comes from combining these 
observations  with the inequalities in  (3.3).
  
  For part (2) we similarly observe that by applying repeated $\mathcal L_{t_0}$ then 
  the positivity of the operator and 
  assumption (5)  we have that 
for all $x \in [1/\alpha, 1]$ 
$$
h_1(x)  \leq
\mathcal L_{t_1} h_1(x) 
\leq 
\mathcal L_{t_1}^2 h_1(x) 
\leq 
\cdots
\leq \mathcal L_{t_1}^n h_1(x) \leq 
\cdots .
\eqno(3.4)
$$
As in part (1), by Lemma \ref{ruelle} we see that
 $$\lim_{n \to +\infty } \|\mathcal L_{t_1}^n h_1\|_\infty^{1/n} = \lambda(t_1)$$
  and, again,  trivially 
  $\lim_{n \to +\infty } \| h_1\|_\infty^{1/n} = 1$ and so the conclusion comes from  (3.4). 
\end{proof}

\begin{remark}
In Lemma \ref{minmax} it would suffice to consider positive test functions $h_0, h_1: [\alpha,1]] \to \mathbb R^+$ which are only continuous.  However, in our application in the next section the test functions we construct will actually be polynomial, so there is no real loss in assuming in the lemma that these functions are $C^1$.
\end{remark}
\section{Implementation}
We need to convert the theoretical bounds in the previous section into an effective method to numerically estimate $\dim(X)$.
In light of the general  Lemmas \ref{dim}  and \ref{minmax} we have  useful criteria to check whether $\dim(X)$ lies in a given interval $[t_0, t_1]$.   

We begin with the following reformulation.

 \begin{proposition}[Criteria for  bounds on $\dim(X)$]\label{crit}
Let  $t_0 < t_1$.  A sufficient condition that 
$$t_0 < \dim(X) < t_1$$
 is the existence of  continuous  functions
$h_0, h_1: [1/\alpha,1] \to \mathbb R$ such that 
$$\sup_{-1/\alpha\leq x \leq 1} \frac{\mathcal L_{t_0}h_0(x)}{h_0(x)} <1 
\hbox{ and }
\sup_{-1/\alpha\leq x \leq 1} \frac{\mathcal L_{t_1}h_1(x)}{h_1(x)}>1.\eqno(4.1)$$
\end{proposition}
\medskip

To apply  Proposition \ref{crit} we  need to find appropriate  functions $h_0$, $h_1$ 
and values $t_0$, $t_1$. We address these issues in the next subsections.

\subsection{Interpolation}
Assume we have candidate values for $t_0 < t_1$.
To choose the functions $h_0$, $h_1$
we proceed  using a simple collocation method.

\begin{definition}
Fix a natural number $m \geq 2$.
\begin{enumerate}
\item
Consider the  {\it Chebychev points} 
$$x_m  = \left( \frac{1 + 1/|\alpha|}{2} \right)\cos \left(\frac{n \pi}{2m} \right)
+ \left(\frac{1}{\alpha} + 1\right) 
 \hbox{ for $0 \leq n \leq m$}
$$
scaled to lie in the interval  $[1/\alpha, 1]$.
\item
Consider the {\it Lagrange polynomials}  $\ell_m:[1/\alpha, 1] \to \mathbb R$  
defined by
$$
\ell_k(x) = \frac{\prod_{n\neq k} (x-x_m)}{\prod_{n \neq k} (x_k-x_m)} \hbox{ for $0 \leq n \leq m$}
$$
scaled to be defined  on  the interval  $[1/\alpha, 1]$.
\end{enumerate}
\end{definition}
In particular,  we have the useful property 
$$
\ell_k(x_m) 
=
\begin{cases}
1 & \hbox{ if } k = n\\
0 & \hbox{ if } k \neq n\\
\end{cases}
$$ 
for $0 \leq k,n \leq m$.
This allows us to  proceed as follows.  

\begin{enumerate}
\item[(i)]
Let $[t_0, t_1]$ be an interval in which we what to verify that the value $\dim(X)$ can be found.
\item[(ii)]
For each of the choices $i=1,2$ we can associate the $(m+1)\times (m+1)$-matrices
$$
M _i= \left( (\mathcal L_{t_i}\ell_k)(x_m) \right)_{k,n=0}^m 
$$
\item[(iii)]
Providing $m$ is sufficiently large,  there is a  maximal 
positive eigenvalue $\lambda_i > 0$ and a corresponding 
right eigenvector 
 $\underline v^i = (v^i_0, \cdots,  v^i_m)$ has strictly positive entries, i.e., 
 $\lambda_i \underline v^i = \underline v^i M_i$  and $v_i > 0$ for $0 \leq i \leq m$.
\item[(iv)]
Associate the two polynomial  functions
$$
h_{i}(x) = \sum_{n=0}^m v_n^i \ell_n(x) \hbox{ for } x \in [1/\alpha,1]  \quad (\hbox{for }i=0,1).
$$
\item[(v)] Finally, if 
$$\sup_{1/\alpha\leq x \leq 1} 
\frac{(\mathcal L_{t_0}h_0)(x)}{h_0(x)} <1\ \hbox{ and }
\inf_{1/\alpha\leq x \leq 1} \frac{(\mathcal L_{t_1}h_1)(x)}{h_1(x)} >1 \eqno(4.2)$$ 
then we can deduce that $t_0 < \dim(X) < t_1$. 
\end{enumerate}

\begin{remark}[Effectiveness of the algorithm]
If $t_0 < \dim X < t_1$ but (4.2) isn't satisfied, then by increasing $m$ it can be achieved (as claimed above  in (iii)).
More precisely, providing $m$ is sufficiently large  and $t_0 < \dim(X) < 1$ 
(or equivalently $\lambda(t_0) > 0 > \lambda(t_1)$)
we  have  that
\begin{enumerate}
\item[(a)] for $h_i(x) > 0$ ($i=0,1$); and
\item[(b)] the appropriate inequality in (9) holds.  
\end{enumerate}

This follows by analytic perturbation theory, but for the present concrete setting if suffices to check this conclusion empirically.   

\end{remark}

\subsection{Bisection method}
In order to obtain candidate values of $t_0 < t_1$ which give upper and lower bounds on $\dim X$ we proceed by  a simple bisection method.   This generates a sequence of improving bounds 
$t_0^{(k)} < \dim X < t_1^{(k)}$, for $k\geq 0$ where $|t_0^{(k)} - t_1^{(k)}| \to 0$ as $k \to +\infty$.

\begin{enumerate}
\item We begin by choosing obvious upper and lower bounds $t_0^{(1)} < \dim X < t_1^{(1)}$ which are trivially valid.  For example,  we can let  $t_0^{(1)} = 3/10$ and $t_1^{(1)} = 7/10$.
\item
We proceed to construct the sequences $t_0^{(k)}$ and  $t_1^{(k)}$ ($k \geq 0$) inductively.
Given $t_0^{(k)} < \dim X < t_1^{(k)}$ (for $k \geq 1$) we can provisionally set 
$$s = \frac{1}{2} (t_1^{(k)} - t_0^{(k)})$$
\item We can associate to $\mathcal L_s$ a function $h_s$ using  the interpolation method described (ii), (iii) and (iv) above.  Typically, we have one of the following two cases.
\begin{enumerate}
\item  If we have 
$$\sup_{1/\alpha\leq x \leq 1} 
\frac{(\mathcal L_{s}h_s)(x)}{h_s(x)} <1$$
then set $t_0^{(k+1)} = s  $ and $t_1^{(k+1)} = t_1^{(k)} $.
\item If we have 
$$\inf_{1/\alpha\leq x \leq 1} \frac{(\mathcal L_{s}h_s)(x)}{h_s(x)} >1$$ 
then set $t_1^{(k+1)} = s  $ and $t_0^{(k+1)} = t_0^{(k)} $.
\end{enumerate}
\end{enumerate}
This produces   a sequence of shrinking  intervals $[t_0^{(k)},t_1^{(k)}]$ ($k \geq 0$) such that:
\begin{enumerate}
\item
$t_0^{(k)} \leq \dim(X)  \leq t_1^{(k)}]$; and 
\item
$|t_1^{(k)} - t_0^{(k)}| = 2^{-k+1}|t_1^{(1)} - t_0^{(1)}|$. 
\end{enumerate}
\begin{remark}
 In neither of the hypotheses  in (a) and (b) holds then it will  be necessary to increase the value of $m$, as mentioned in the previous remark.
\end{remark}

\subsection{Numerical estimates}
We can use the  approximation to   $g$  in \cite{mathar}.  
In particular, this gives estimates for 
the first $156$ terms in the 
 Chebychev  polynomial expansion, each coefficient presented in excess 
of $150$ decimal places.

\begin{example}[Estimate to $2$ decimal places]
If we only require a more modest rigorous bound  $0.53 \leq \dim(X) \leq 0.54$ then we can apply the 
 criteria for dimension bounds on $X$ using simpler explicit polynomials 
\begin{enumerate}
\item
For 
$t_0= 0.53$ and 
$$h_0(x) = -2933 + 148 x$$ we can  estimate 
$$
\min \frac{\mathcal L_{t_0}h_0(x)}{h_0(x)} = 1.00023 \ldots
$$
\item 
For 
$t_1= 0.54$ and 
$$
\begin{aligned}
h_1(x) &= -293406 + 15032 x - 9718 x^2 + 4200 x^3 \cr
&\qquad- 2391 x^4 + 
  1336 x^5 - 773 x^6 + 432 x^7
  \end{aligned}
  $$ we can estimate
$$
\max \frac{\mathcal L_{t_1}h_1(x)}{h_1(x)} = 0.999964  \ldots
$$
\end{enumerate}
In particular, we can deduce 
from Proposition \ref{crit}
that $0.53 \leq \dim(X) \leq 0.54$.
\end{example}

\begin{example}[Estimate to $60$ decimal places]
Working to $100$ decimal places  we can  choose $m=300$ and  implement the approach described in the previous section.
Starting from 
$t_0^{(1)}=3/10$ and $t_1^{(1)}=7/10$
we can apply the bisection method  until we get suitable values.
 For 
$$
\begin{aligned}
t_0 = &
0.5380451435 8054991167 141556737 4986292737\cr
&\qquad 9649658778 5696090719 1\cdots \cr
\end{aligned}
$$
we can   associate a suitable polynomials $h_0$ 
and  then estimate 
$$
\begin{aligned}
\min \frac{\mathcal L_{t_0}h_0(x)}{h_0(x)} = &1.0000000000000000000000000000000000000000000000\cr
&\qquad 00000000000002043024411083008731521919696166 \ldots
\end{aligned}
$$
\item For
$$
\begin{aligned}
t_1 = &
0.538045143580549911671415567374986292737\cr
&\qquad964965877856960907193
\end{aligned}
$$
we can   associate a suitable polynomials $h_0$ 
and then estimate 
$$
\begin{aligned}
\max \frac{\mathcal L_{t_0}h_0(x)}{h_0(x)} = & 0.9999999999999999999999999999999999999999999999\cr
&\qquad 999999999999995657268412268496192385300641278 \ldots
\end{aligned}
$$
\end{example}

\begin{remark}
The above calculations were carried out using Mathematica using its internal routines  and error estimates.  However, for complete confidence in the estimate it would probably be best to use a more transparent programme.
\end{remark}

\section{Final comments}
Since the aim of this note is to explain one way in which  precise estimates on $g$ can be converted into good bounds on $\dim(X)$ it is useful to review the sources of loss of accuracy and the restrictions on this approach.

\subsection{Approximation of the transfer operators}

The contractions $T_0, T_1: [1/\alpha,1] \to [1/\alpha,1]$ used in the definition of the transfer operators $\mathcal L_t$ are defined in terms of $g$ (and $\alpha$) .
If  $g(z)$ 
has a uniform polynomial approximation   by a polynomial  $\tilde g(z)$ 
 on the  disk  $ [1/\alpha,1] \subset D = \{z \in \mathbb C \hbox{ : } |z| < \sqrt{8}\}$ (see \cite{Lanford})
then the corresponding contractions $T_i$ have comparable
approximations   $\widetilde T_i$ ($i=1,2$)  on  $[1/\alpha, 1]$
(and similarly for their derivatives by Cauchy's theorem).
The associated operator $\widetilde {\mathcal L}_t$ applied to the  approximating polynomial  $h: [\alpha, 1] \to \mathbb R$ (of degree $m$, say)  has  a uniform bound
 $$
 \begin{aligned}
 &\|\widetilde {\mathcal L}_t h -  {\mathcal L}_t h\|_\infty \cr
 &\leq \sum_{i=1}^2 \|(T_i')^t - (S_i')^t\|_\infty \|h\circ T_i\|_\infty + 
 \|(T_i')^t \|_\infty \|h\circ T_i - h\circ S_i\|_\infty
 \end{aligned}\eqno(5.1)
 $$
 and by the mean value theorem:
 $
 \|h\circ T_i - h\circ S_i \|_\infty \leq \| h'\|_\infty \|T_i - S_i\|_\infty.
 $
 Let  $\pi$ be the  projection onto polynomials of degree $m$ then we can take $h_0$ to be the maximal eigenfunction of $\mathcal L \pi$ 
(and  $h$   to be the  maximal eigenfunction for $\mathcal L_t$) then 
by  \cite{bs} 
 there exists $C>0$ and $0 < \theta < 1$ such that 
 $\|\mathcal L_t - \mathcal L_t \pi\|_\infty \leq C \|\mathcal L_t\| \theta^m$
 on $D$ for $m \geq 1$.  In particular, 
 $\|h - h_0\| \to 0$ on $D$ as $m \to +\infty$ and  bound 
$ \|h'\|_\infty$ on $[1/\alpha,1]$  by  Cauchy's theorem.   This gives boundes on (5.1).



\subsection{Checking the minimax bounds} In \S 4 
we need to effectively   estimate the  supremum
or infimum 
 of functions $\frac{{\mathcal L_t h}}{h}(x)$ over $[1/\alpha,1]$.
To this end the following observation is helpful.  
\begin{lemma}\label{single}
If $h: D \to \mathbb R$ is the collocation function associated to  $m$ then 
$$\sup_{1/\alpha \leq x \leq 1} \left|\left(\frac{\mathcal L_t h}{h}\right)'(x)\right| \to 0 \hbox{ as } m \to +\infty.$$
\end{lemma}


\begin{proof}
By the quotient rule
$$
\left(
\frac{\mathcal L_t h}{h}
\right)'(z)
=
\frac{(\mathcal L_t h)' (z) h(z) - (\mathcal L_t h)(z) h' (z)}{h(z)^2}
 \hbox{ for $z \in D$.}
$$
As above, by  \cite{bs} we have  that  $\|\mathcal L_t  - \mathcal L_t \pi\|_\infty \to 0$ on $D$ as $m \to +\infty$ 
which gives that   $\|h - h_0\| \to 0$.    Thus 
$\sup_{\alpha \leq x \leq 1} |h'(x) - h_0'(x)| \to 0$ by Cauchy's theorem 
 and similarly $\sup_{\alpha \leq x \leq 1} |(\mathcal L_t h)'(x) - (\mathcal L_th_0)'(x)| \to 0$.  Finally, since 
$\inf_{x\in I}h_m(x) \to \inf_{x\in I}h(x) > 0$ as $n \to +\infty$, 
combining these results gives the conclusion.
\end{proof}

\end{document}